\documentclass[12pt, a4paper]{amsart}

\usepackage{amsmath,amssymb}
\theoremstyle{plain}
\newtheorem{theorem}{Theorem}[section]
\newtheorem{lemma}[theorem]{Lemma}

\newtheorem{proposition}[theorem]{Proposition}
\theoremstyle{definition}
\newtheorem{remark}[theorem]{Remark}

\newtheorem{definition}[theorem]{Definition}

\numberwithin{equation}{section}

\def\R{\mathbb R}

\def\la{\lambda}

\begin{document}

\baselineskip=15pt    

\title{Boundedness of the Hilbert transform on weighted Lorentz spaces} 
\author {Elona Agora, Mar\'{\i}a J. Carro, and Javier Soria} 

\address{E. Agora, Departament de
Matem\`atica Aplicada i~An\`alisi, Universitat de Barcelona, 
 08007 Barcelona, Spain.}
\email{elona.agora@gmail.com}

\address{M. J. Carro, Departament de
Matem\`atica Aplicada i~An\`alisi, Universitat de Barcelona, 
 08007 Barcelona, Spain.}
\email{carro@ub.edu}

\address{J. Soria, Departament de
Matem\`atica Aplicada i~An\`alisi, Universitat de Barcelona, 
 08007 Barcelona, Spain.}
\email{soria@ub.edu}

\thanks{This work was  partially supported by the Spanish Government Grant MTM2010-14946} 

\begin{abstract} We study the boundedness of the Hilbert transform $H$ and the  Hilbert maximal operator $H^*$ on weighted Lorentz spaces $\Lambda^p_u(w)$. We start by giving several necessary conditions that, in particular,  lead us to the complete characterization of the weak-type boundedness of both $H$ and $H^*$,  whenever $u\in A_1$. For the strong-type case, we also get the  characterization of both operators when  $p>1$. Applications to the case of Lorentz spaces  $L^{p,q}(u)$ are presented. 
\end{abstract}

%\date{\today}

\maketitle

\pagestyle{headings}\pagenumbering{arabic}\thispagestyle{plain}

\markboth{Boundedness of the Hilbert transform on weighted Lorentz spaces}{Elona Agora, Mar\'{\i}a J. Carro, and Javier Soria }

\section{Introduction and motivation}

In 1972,  Muckenhoupt \cite{m:m} gave the complete characterization of the boundedness of the Hardy-Littlewood maximal operator $M$ on  weighted Lebesgue spaces $L^p(u)$, with  $p>1$. Recall that 
$$
Mf(x)=\sup_{x\in I}{\frac{1}{|I|}}\int_{I} |f(y)|dy,
$$
where $I$ is an interval of the real line and the supremum is considered over all intervals containing $x\in \R$ and,  for a positive locally integrable function $u$ in $\mathbb R$, $L^p(u)$ is defined as the set of all Lebesgue measurable  functions $f$ such that
$$
||f||_{L^p(u)}=\bigg(\int_{\mathbb R} |f(x)|^p u(x) dx \bigg)^{1/p}<\infty.
$$
The characterization \cite{m:m}  was given in terms of the   $A_p$ condition:
$$
\sup_{I}\left(\frac{1}{|I|}\int_I u(x)dx\right) \left(\frac{1}{|I|} \int_I u^{-1/(p-1)}(x) dx \right)^{p-1}<\infty. 
$$
Similarly, if we consider the weak-type spaces $L^{p, \infty}(u)$ defined by 
$$
||f||_{L^{p, \infty}(u)}=\sup_{t>0} f_u^*(t) t^{1/p}<\infty,
$$
where   $f^*_u$ is the decreasing rearrangement of $f$ with respect to  $u$ (see \cite{bs:bs}), it was also proved in \cite{m:m}  that, if $p>1$,
$$
M:L^p(u)\longrightarrow L^p(u) \qquad\iff\qquad M:L^p(u)\longrightarrow L^{p,\infty}(u), 
$$
and that
$$
M:L^1(u)\longrightarrow L^{1,\infty}(u)
$$
is bounded if and only if $u\in A_1$; that is,
$$
Mu(x)\le C u(x), \qquad a.e. \  x\in\mathbb R. 
$$

Some years later, in \cite{am:am}, Ari\~ no and Muckenhoupt gave the complete characterization of the boundedness of $M$ on some weighted Lorentz spaces. These spaces were defined by Lorentz in \cite{l1:l1}
and \cite{l2:l2} as follows: let $\mathcal M(\mathbb R)$ be the class of Lebesgue measurable functions on $\mathbb R$. Then
$$
\Lambda^{p}_{u}(w) =\left\{f\in\mathcal M(\mathbb R): \, ||f||_{\Lambda^{p}_{u}(w)}=\left( \int_0^{\infty}  (f^*_u(t))^p w(t)dt\right) ^ {1/p}< \infty \right\},
$$
and 
$$
\Lambda^{p,\infty}_{u}(w) =\left\{f\in \mathcal M(\mathbb R): \, ||f||_{\Lambda^{p,\infty}_{u}(w)}=\sup_{t>0} W^{1/p}(t)f^*_u(t)<\infty \right\},
$$
with $u$ and $w$  positive locally integrable functions in $\mathbb R$ and $(0,\infty)$ respectively (these functions will be called weights from now on), and 
$$
W(r)=\int_0^r w(s) ds.
$$

The weighted Lorentz  spaces were deeply studied in \cite{crs:crs}. We list now   some facts that will be important for our purposes.

\begin{proposition}\label{properties}\ 

 \noindent
 (a) $\Lambda^{p}_{u}(w)$ and $\Lambda^{p,\infty}_{u}(w)$  are quasi-normed spaces if and only if $w$ satisfies  the $\Delta_2$ condition ($w\in \Delta_2$); that is, $W(2r)\lesssim W(r)$, $r>0$.

 \noindent
 (b)  Assume that $u\notin L^1(\mathbb R)$, $w\notin L^1(\mathbb R^+)$ and $w\in \Delta_2$.  If  $|g_n|\le |f|$, $f\in \Lambda^p_u(w)$ and  $\,\lim_{n\to \infty}g_n=g$ almost everywhere, then  $\lim_{n\to \infty}||g-g_n||_{\Lambda^p_u(w)}=0$.

 \noindent
 (c)  If $w\in \Delta_2$,   the class of simple functions with support in a set of finite measure,
$$
{\mathcal{S}}_0(u)=\Big\{f \, \in {\mathcal{M}}:\, \text{card}(f(\R))<\infty, \, u(\{f\neq 0\})<\infty\Big\}
$$
is dense in $\Lambda^p_u(w)$.
\end{proposition}

The result proved in \cite{am:am} (see also, for example,  \cite{cgs:cgs}, \cite{cs:cs}) corresponds to  the case $u=1$ and is the following: 
$$
M:\Lambda^{p}(w) \longrightarrow \Lambda^{p}(w)
$$
is bounded if and only if $w\in B_p$; that is, 
$$
 \int_r^{\infty}\left(\frac{r}{t}\right)^pw(t) \, dt \lesssim  \int_0^rw(s) ds, 
$$
for every $r>0$. 
Moreover, if $p>1$ this condition also characterizes the weak-type boundedness:
$$
M:\Lambda^{p}(w) \longrightarrow \Lambda^{p, \infty}(w)
$$
and, if $p\le 1$,  the weak-type boundedness holds if and only  if $w$ is $p$ quasi-concave; that is,  there exists $C>0$ such that, for every $0<r<t<\infty$,
$$
\frac{W(t)}{t^p}\le C \frac{W(r)}{r^p}.
$$
Some important facts about these classes of weights are the following \cite{am:am,   crs:crs, n1:n1}: by definition, we say that $w\in B_{p, \infty}$ if and only if the Hardy operator  
$$
Pf(t)=\frac{1}{t}\int^t_{0}f(s) \, ds,
$$
satisfies that 
$$
P:L^p_{{\rm dec}}(w) \rightarrow  L^{p,\infty}(w)
$$
is bounded, where  $L^p_{{\rm dec}}(w)$ is the set of decreasing functions on $L^p(w)$. Then: 

\begin{enumerate}

\item[{(i)}] If $p>1$,  $B_p=B_{p, \infty}$.

\item[{(ii)}] For every $p\le 1$, $w\in B_{p, \infty}$ if and only if $w$ is $p$ quasi-concave. 

\item[{(iii)}] For every $p>0$,  $w\in B_{p}$ if and only if
 $$
P:L^p_{{\rm dec}}(w) \rightarrow  L^{p}(w)
$$
is bounded. 
\end{enumerate}
\medskip

Recently, the complete characterization of the boundedness of 
$$
M:\Lambda^{p}_u(w) \longrightarrow \Lambda^{p}_u(w)
$$
was proved  in \cite{crs:crs} by means of the following condition: there exists $q\in(0,p)$ such that,  for every finite family of intervals $(I_j)_{j=1}^J$, and every family of measurable sets $(S_j)_{j=1}^{J}$, with $S_j\subset I_j$, for every $j$,  we have that
\begin{equation}\label{maximal}
\frac{W\left(u\left(\bigcup_{j=1}^J I_j\right)\right)}{W\left(u\left(\bigcup_{j=1}^J S_j\right)\right)}
       \lesssim  \max_{1\leq j\leq J} \left(\frac{|I_j|}{|S_j|}\right)^q, 
\end{equation}
where, for a measurable set $E$, we write $u(E)$ to denote $\int_E u(x) dx$, 
and  as usual, we use the symbol $A\lesssim B$ to indicate that there exists a universal constant $C>0$, independent of all important parameters, such that $A\le C B$ and $A\approx B$ indicates that both $A\lesssim B$  and $B\lesssim A$ hold. This is   the standard notation that we will also use throughout this paper.

Similar to the case of the maximal operator $M$, there is also a corresponding theory for the Hilbert transform (or, more generally, for singular integral operators). The Hilbert transform is defined by:
$$
Hf(x)=\frac{1}{\pi} \lim_{\varepsilon\to 0^+} \int_{|x-y| > \varepsilon} \frac{f(y)}{x-y}\,dy,
$$
whenever this limit exists almost everywhere. We know that this is the case for simple functions
$$
{\mathcal{S}}(u)=\Big\{f \, \in  {\mathcal{M}}:\, \text{card}(f(\R))<\infty\Big\},
$$
or for $\mathcal{C}^\infty$ functions with compact support,  ${\mathcal{C}}^{\infty}_c$. We shall also consider the  Hilbert maximal operator 
$$
H^*f(x)=\frac{1}{\pi} \sup_{\varepsilon> 0} \left|\int_{|x-y| > \varepsilon} \frac{f(y)}{x-y}\,dy\right|. 
$$

Our main goal  is to study the boundedness of the Hilbert transform and of the Hilbert maximal operator on weighted Lorentz spaces
\begin{equation} \label{Hilbertantecedents}
H, H^*: \Lambda^p_u(w) \longrightarrow\Lambda^p_u(w)
\end{equation}
and also its weak version $H, H^*:\Lambda^p_u(w) \longrightarrow\Lambda^{p, \infty}_u(w)$. We shall give first some general necessary conditions and obtain,  as a consequence,  the characterization of (\ref{Hilbertantecedents}) and its weak-type version under the hypothesis that $u\in A_1$.  

As in  the case of the Hardy-Littlewood maximal operator, the cases $w=1$ and $u=1$ are already known:

\noindent
(i) 
  If $w=1$, \eqref{Hilbertantecedents} is equivalent to the fact that 
$$
H:L^p(u) \to L^p(u)
$$
is bounded and this problem was completely solved by Hunt, Muckenhoupt,  and Wheeden in \cite{mhw:mhw}. An alternative proof was provided in \cite{cf:cf} by Coifman and Fefferman and the condition is  that $u\in A_p$, if $p>1$. This property also characterizes the weak-type boundedness:
$$
H, H^*:L^p(u) \to L^{p,\infty}(u),
$$
and, if $p=1$, 
$$
H, H^*:L^1(u) \to L^{1,\infty}(u)
$$
is bounded if and only if $u\in A_1$. Moreover, it is known that if 
\begin{equation}\label{ff}
H, H^*:L^p(u) \to L^{p,\infty}(u)
\end{equation}
is bounded, then necessarily $p\ge 1$; that is,  there are no weights $u$ for which the boundedness holds for some $p<1$. 
This follows from the fact that (\ref{ff}) implies that, for every measurable set $E\subset I$, 
$$
\frac{u(I)}{u(E)} \le C \bigg( \frac {|I|}{|E|}\bigg)^p
$$
and if $p<1$, the Lebesgue differentiation theorem would imply that $u=0$  (see \cite{mhw:mhw}).

\noindent
(ii) On the other hand, if $u=1$, the characterization of \eqref{Hilbertantecedents} is equivalent to the boundedness of
$$
H, H^*:\Lambda^p(w) \longrightarrow  \Lambda^p(w),
$$
given by Sawyer in \cite{s:s}.  A simplified description of the class of weights that characterizes the above boundedness   is $B_p\cap B^*_{\infty}$,  where a weight  $w\in  B^*_{\infty}$ if 
\begin{equation}\label{binfty}
\int_0^r \frac{1}{t}\int_{0}^t w(s)ds \ dt \lesssim \int_{0}^r w(s)ds,
\end{equation}
for all $r>0$ (see \cite{n:n}). 

If we consider the weak-type boundedness 
$$
H, H^*:\Lambda^p(w) \longrightarrow  \Lambda^{p, \infty}(w),
$$
then the complete characterization is $w\in B_{p, \infty}\cap B^*_{\infty}$.
 \medskip
 
The paper is organized as follows: in Section 2, we shall present several necessary conditions to have  the weak-type boundedness of $H$ on $\Lambda^p_u(w)$. As a consequence, we prove in Section 3 two of the  main results of this paper,  Theorems \ref{a1} and \ref{a2}, which give the  characterization of the boundedness of $H$ under the assumption that $u\in A_1$,  namely,  for every $p>0$, 
$$
H: \Lambda^p_u(w) \to \Lambda^{p, \infty}_u(w)  \Longleftrightarrow  w\in B_{p,\infty}\cap B^*_{\infty}, 
$$
and, if $p>1$, 
$$
H: \Lambda^p_u(w) \to \Lambda^{p}_u(w)  \Longleftrightarrow  w\in B_{p}\cap B^*_{\infty} .
$$

Moreover, this characterization also works for the Hilbert maximal operator $H^*$ (see Remark \ref{maxim}). 

\smallskip

In order to avoid trivial cases, we shall assume that the weights $u$ and $w$ satisfy the following condition:

$$
W\left(\int_{-\infty}^{+\infty}u(x)\,dx\right)>0. 
$$

Finally, we  recall the definition of  the associated space of $\Lambda^p_u(w)$ (see~\cite{bs:bs}) since it will appear later on. 

 \begin{definition}
The  \textit{associate space}  of  $\Lambda^p_u(w)$  is defined as the set of all measurable functions $g$ such that 
$$
||g||_{(\Lambda^p_u(w))'}:= \sup_{f\in \Lambda^p_u(w)}\frac{\left|\displaystyle\int_{\mathbb R}f(x)g(x)u(x) dx\right| }{||f||_{\Lambda^p_u(w)}}<\infty.
$$
\end{definition}

In \cite{crs:crs},  these spaces were  given 
in terms of the Lorentz spaces $\Gamma$ defined as follows.

\begin{definition}
If $0<p<\infty$, 
$$
\Gamma_{u}^p(w)= \left\{f\in {\mathcal{M}}: ||f||_{\Gamma^p_u(w)}=\left( \int_0^{\infty}(f^{**}_{u}(t))^pw(t)dt \right)^{1/p}< \infty \right\}, 
$$
where $f^{**}_{u}(t)=P f_u^*(t)$, and 
$$
\Gamma_{u}^{p,\infty}(w)= \left\{f\in {\mathcal{M}}: ||f||_{\Gamma^{p,\infty}_u(w)}=\sup_{t>0} W^{1/p}(t) f^{**}_{u}(t)< \infty \right\}.
$$
\end{definition}

\begin{theorem} \label{associate} \cite{crs:crs}
The associate spaces of the Lorentz spaces are described as follows:

\noindent
(i) If $p\leq 1$, then
$$
(\Lambda^p_u(w))'= \Gamma^{1,\infty}_u(\tilde{w}),
$$
where $\widetilde{W}(t)=tW^{-1/p}(t)$, $t>0$.
\noindent
(ii) If $1<p<\infty$ and  $w\notin L^1(0, \infty)$, 
$$
(\Lambda^p_u(w))'= \Gamma^{p'}_u(v),
$$
where $v(t)=t^{p'}W^{-p'}(t)w(t)$, $t>0$.
\end{theorem}

\noindent
{Acknowledgement:} We want to express our gratitude to the referee for all his/her comments and suggestions.

\section{Necessary conditions}

We start by considering some necessary conditions for the boundedness of the Hilbert transform.
 \begin{theorem}\label{NC}

Let $0<p<\infty$ and let us assume that the operator $H$ is well defined on $\Lambda^p_u(w)$ and that
$$
H:\Lambda^p_u(w) \longrightarrow\Lambda^{p, \infty}_u(w)
$$
is bounded. Then, the following conditions hold: 

\noindent
(a)   $u\not\in L^1(\R)$ and $w\not\in L^1(\R^+)$.

\noindent
(b) For every interval $I$ and every measurable set $E$   such that $E\subset I$, we have that
\begin{equation} \label{quasiconcrara}
\frac{W(u( I ))}{W(u( E ))}\lesssim \left(\frac{|I|}{|E|}\right)^p.
\end{equation}
In particular, $W\circ u$ satisfies the doubling property; that is,  $W(u(2I))\lesssim  W((u(I)))$,
for all intervals $I\subset\R$, where $2I$ denotes the interval with the
same center than $I$ and double size-length. 

\noindent
(c) $W$ is $p$ quasi-concave. 

\noindent
(d) For every  interval $I$, 
\begin{equation} \label{conocidodiagonal}
||u^{-1}\chi_{I}||_{(\Lambda^p_u(w))'}||\chi_{I}||_{\Lambda^p_u(w)} \lesssim |I|. 
\end{equation}
\end{theorem}

\begin{remark}

One could think that, in the case $p>1$,   the boundedness $H: \Lambda^p_u(w)\to \Lambda^p_u(w)$ holds if both the boundedness $H:L^p(u) \to L^p(u)$ (characterized by the condition $A_p$) and
the boundedness $H:\Lambda^p(w) \to \Lambda^p(w)$ (characterized by $w\in B_p \cap B^*_{\infty}$) hold. However, we prove with the following examples of $u$ and $w$  that, in general, the
conditions $u\in A_p$ and $w\in B_p \cap B^*_{\infty}$ are not sufficient for the boundedness of $H$ on $\Lambda^p_u(w)$, for $p>1$. 
\medskip

If the Hilbert transform is bounded on $\Lambda^p_u(w)$, with  $u(x)=|x|^{k}$ and $w(t)=t^l$ with $k,l > -1$,  then (\ref{quasiconcrara}) implies that  
\begin{equation}\label{contra}
(k+1)(l+1)\leq p.
\end{equation}
But, if we choose  $p$ and $k=l$ such that  $\sqrt{p}<k+1 <p$, then $u(x)=|x|^{k}\in A_p$ and $w(t)=t^{k} \in B_p\cap B^*_{\infty}$.
However, such $k=l>-1$ do  not satisfy  \eqref{contra},  since $p<(k+1)^2$. Hence, in this case, the Hilbert
transform is not bounded on $\Lambda^p_u(w)$.

\end{remark}

\

\centerline{\bf Proof of Theorem \ref{NC}}

\medskip

We will split the proof into several propositions. 
To prove {\it (a)} we  need the following lemmas:

\begin{lemma}
Let $a,b\in \R$ and let $\lambda>0$. Then 
\begin{equation}\label{fd de la H}
u\left(\left\{x:\, |H\chi_{(a,b)}(x)|>\lambda \right\}\right)= \int_{a-\psi(\la)}^{a+\varphi(\la)} u(s) \,ds
           +\int_{b-\varphi(\la)}^{b+\psi(\la)} u(s)\,ds,
\end{equation}
where
$$
\varphi(\la)=(b-a)\frac{1}{1+e^{\pi\lambda}}  \qquad{and}\qquad \psi(\la)=(b-a)\frac{1}{e^{\pi\lambda}-1}.
$$
\end{lemma}

\begin{proof}
Since it is known (see \cite{g:g}) that
$$
 \displaystyle{H\chi_{(a,b)}(x)=\frac{1}{\pi}\log\frac{|x-a|}{|x-b|}},
$$
the proof follows by easy computations. 
\end{proof}

\begin{lemma} \label{necSteinWeiss}
Let $0<p<\infty$. If the Hilbert transform satisfies that
\begin{equation} \label{weak 1}
\|H\chi_{(0,b)}\|_{\Lambda^{p, \infty}_u(w)} \lesssim \|\chi_{(0,b)}\|_{\Lambda^{p}_u(w)},
\end{equation}
for all $b>0$, then for every $\nu \in (0,1]$, 
\begin{equation}\label{log}
\sup_{b>0}\frac{W\left(\int_{-b\nu}^{b\nu} u(s)\,ds\right)}{W\left(\int_{-b}^b u(s)\,ds\right)}\lesssim
               \left(1+ \log\frac{1}{\nu}\right)^{-p}.
\end{equation}

\end{lemma}

\begin{proof}
Let $b>0$.  By hypothesis we have that
$$
\sup_{\lambda>0} W\big(\,u(\{x:\, |H\chi_{(0,b)}(x)|>\lambda\})\,\big)\lambda^p\lesssim W\left(\int_0^bu(s)\,ds\right),
$$
which, applying \eqref{fd de la H}, is equivalent to

$$
\sup_{\lambda>0} W\left( \int_{-\psi(\la)}^{\varphi(\la)} u(s)\,+\int_{b-\varphi(\la)}^{b+\psi(\la)} u(s)\,ds\right)
      \lambda^p\lesssim W\left(\int_0^bu(s)\,ds\right).
$$
Then, we necessarily obtain that, for
every $\lambda>0$,
$$
	  W\left(\int_{\frac{b}{1-e^{\pi\lambda}}}^{\frac{b}{1+e^{\pi\lambda}}} u(s)\,ds\right)\lambda^p\lesssim
	W\left(\int_0^b u(s)\,ds\right).
$$
Since $\displaystyle{\frac{b}{1-e^{\pi\lambda}}< \frac{-b}{1+e^{\pi \la}}<0<\frac{b}{1+e^{\pi\lambda}}}$,  we obtain that 

$$
	  W\left(\int_{\frac{-b}{1+e^{\pi\lambda}}}^{\frac{b}{1+e^{\pi\lambda}}} u(s)\,ds\right)\lambda^p\lesssim
	W\left(\int_0^b u(s)\,ds\right).
$$
Writing  $\nu=\frac{1}{1+e^{\pi\la}}$ we get

$$
\sup_{b>0}\frac{W\left(\int_{-b\nu}^{b\nu} u(s)\,ds\right)}{W\left(\int_{-b}^b u(s)\,ds\right)}\lesssim
\left(\log \frac{1-\nu}{\nu}\right)^{-p}\approx  \left(1+ \log\frac{1}{\nu}\right)^{-p}, 
$$
for every $\nu \in(0,1/2)$.
On the other hand since,  for every $\nu \in(0,1]$, 

$$
\sup_{b>0}\frac{W\left(\int_{-b\nu}^{b\nu} u(s)\,ds\right)}{W\left(\int_{-b}^b u(s)\,ds\right)}\leq 1, 
$$
we obtain that (\ref{log}) holds. 
\end{proof}

\begin{remark}
It is known (see \cite{crs:crs})  that if the  Hardy-Littlewood maximal operator  satisfies that $M:\Lambda^p_u(w) \to \Lambda^{p, \infty}_u(w)$ is bounded, then  $u$ is necessarily non-integrable, whereas there are no integrability restrictions on $w$.   However, we shall prove that
if the Hilbert transform satisfies (\ref{necSteinWeiss}) then both $u$ and
$w$ are non-integrable.

If $u=1$, we recover a well-known result proved by Sawyer in \cite{s:s}, which states  that
the boundedness of the Hilbert transform on classical Lorentz spaces $H:\Lambda^p (w) \to \Lambda^{p} (w)$ implies the non-integrability of $w$.
\end{remark}

As a consequence of the following proposition we obtain the proof of Theorem \ref{NC} {\it (a)}.

\begin{proposition}\label{no estan en L1}
If the Hilbert transform satisfies \eqref{weak 1}, 
then $u\not\in L^1(\R)$ and $w\not\in L^1(\R^+)$.
\end{proposition}

\begin{proof}
Since $w$ is locally integrable, it is enough to prove that
\begin{equation}\label{objetivo}
W\left(\int_{-\infty}^{+\infty}u(x)\,dx\right)=\lim_{t\to\infty}W\left(\int_{-t}^{\,t}u(x)\,dx\right)=\infty.
\end{equation}
Suppose that this limit is a finite number $\ell>0$.
Since, by Lemma~\ref{necSteinWeiss} we  have that, there exists $C>0$ such that, for all $\nu \in (0,1]$,
$$
\sup_{b>0}\frac{W\left(\int_{-b\nu}^{b\nu} u(s)\,ds\right)}{W\left(\int_{-b}^b u(s)\,ds\right)}\leq C
    \left(\log\frac{1}{\nu}\right)^{-p}, 
$$
 taking $\nu>0$ small enough satisfying  $C \left(\log \frac{1}{\nu}\right)^{-p}<1/2$, we obtain that 
$$
\lim_{b\to\infty}\frac{W\left(\int_{-\nu b}^{\nu b} u(s)\,ds\right)}{W\left(\int_{-b}^{b} u(s)\,ds\right)}\leq
   \sup_{b>0}\frac{W\left(\int_{-\nu b}^{\nu b} u(s)\,ds\right)}{W\left(\int_{-b}^{b} u(s)\,ds\right)}\leq \frac{1}{2}. 
$$
Since we also have that 

$$
\lim_{b\to\infty}\frac{W\left(\int_{-\nu b}^{\nu b} u(s)\,ds\right)}{W\left(\int_{-b}^{b} u(s)\,ds \right)}
=\frac{\ell}{\ell}=1, 
$$
we get a contradiction, and therefore,   \eqref{objetivo} holds.
\end{proof}

The following proposition gives us as a consequence the proof of Theorem \ref{NC} {\it (b)}.

\begin{proposition}\label{necesaria geometrica}
Let $0 < p <\infty$ and assume that, for every measurable set $F$,
$$
||H\chi_F||_{\Lambda^{p, \infty}_u(w)}\lesssim ||\chi_F||_{\Lambda^{p}_u(w)}.
$$
Then, \eqref{quasiconcrara}  holds.

\end{proposition}

\begin{proof}
Let $f$ be a non-negative function supported in $I$. Let  $I'$ be an interval of the same size touching $I$.
If $x\in I'$ we  have that 

$$
|Hf(x)|=\left|\int_\R \frac{f(y)}{x-y}\,dy\right| = \left|\int_{I} \frac{f(y)}{x-y}\,dy\right| \geq\frac{1}{2|I|}
   \int_{I} f(y)\,dy.
$$

If  $f=\chi_{E}$ with $E\subset I$,   and $x\in I'$, then $\frac{|E|}{2|I|}\leq |Hf(x)|$ and hence  if $\la\leq \frac{|E|}{2|I|}$, we have that $I'\subseteq \{x:\, |Hf(x)|>\la\}.$
Therefore
\begin{align*}
W(u( I'))&\leq W(u(\{x:\, |Hf(x)|>\la\}))
\lesssim \frac{1}{\la^p}\int_0^\infty (\chi_{E})^*_u(t)w(t)\,dt\\
&\approx\frac{1}{\la^p}W(u(E)).
\end{align*}
 As the above inequality holds for every $\la\leq \frac{|E|}{2|I|}$, we obtain that 

$$
\frac{W(u(I'))}{W(u(E))}\lesssim  \left(\frac{|I|}{|E|}\right)^p.
$$
So, it only remains to prove that we can replace $I'$ by the interval $I$. In fact,  the quantities $W(u( I' ))$ and $W(u( I ))$ are comparable, since taking  $E=I$ we get 
$
W(u( I')) \lesssim \  W(u ( I )),
$
and  interchanging the roles of $I$ and $I'$ we get the converse inequality 
$
W(u( I )) \lesssim \  W (u ( I')). 
$
\end{proof}

The proof of Theorem \ref{NC} {\it (c)} is a consequence of (\ref{quasiconcrara})  and it was proved in \cite[Lemma 3.3.1]{crs:crs}. 

\medskip

\begin{proposition} \label{dualidad}
Let $0 < p <\infty$ and let us assume that  $H:\Lambda^p_u(w) \to \Lambda^{p, \infty}_u(w)$ is bounded. Then  \eqref{conocidodiagonal} holds.

\end{proposition}

\begin{proof}
Let $I$ and $I'$ be as in Proposition \ref{necesaria geometrica}. Then we have already seen that  if   $f$ is supported in $I$, $f_I=\int_I f(x)dx$ and $\lambda\leq \frac{f_I}{2|I|}$, we have that 
$I'\subseteq \{x:\, |Hf(x)|>\lambda\}$.  Therefore
\begin{align*}
W^{1/p}(u( I'))&\leq W^{1/p}(u(\{x: |Hf(x)|> \lambda\}))\leq C \frac{1}{\lambda} ||f||_{\Lambda^p_u(w) },
\end{align*}
and since this holds for every 
$\lambda\leq \frac{f_I}{2|I|}$, we obtain
$$
\left(\frac{f_I}{||f||_{\Lambda^p_u(w) }} \right){W^{1/p}(u(I'))} \lesssim |I|.
$$

Considering the supremum over all $f\in \Lambda^p_u(w)$ and taking into account that
$$
f_I= \int f(x) (u^{-1}(x)\chi_I(x)) u(x)dx, 
$$
we get that
$$
||u^{-1}\chi_I||_{(\Lambda^p_u(w) )'} {W^{1/p}(u(I'))} \lesssim |I|. 
$$
Since
$$
||\chi_I||^{p}_{\Lambda^p_u(w) }= W(u(I))\leq W(u(3I'))\leq c W(u(I')), 
$$
we obtain the result. 
\end{proof}

And this concludes the proof of our main Theorem \ref{NC}. Let us now analyze in more detail  condition  (\ref{conocidodiagonal}).

\begin{proposition}

\ 

\noindent
i)  If $p\leq 1$,  condition (\ref{conocidodiagonal}) is equivalent to condition (\ref{quasiconcrara}).

\noindent
ii) If $p>1$,  condition \eqref{conocidodiagonal} is equivalent to the following: for every interval $I$, 
\begin{equation} \label{GenHardyNec}
\left(\int_0^{u(I)} \left( \frac{\phi_I(t)}{W(t)} \right)^{p'} w(t)dt\right)^{1/p'} \lesssim \frac{|I|}{W^{1/p}(u(I))},
\end{equation}
where 
\begin{equation}\label{fi}
\phi_{I}(t)=\sup\{|E|: E\subset I, u(E)=t\}.
\end{equation}

\end{proposition}

\begin{proof} {\it i)} If $p\leq 1$,   condition \eqref{conocidodiagonal} is equivalent to 
$$
||u^{-1}\chi_{I}||_{\Gamma^{1,\infty}_u(\tilde{w})}||\chi_{I}||_{\Lambda^p_u(w)}\lesssim |I|,
$$
and by  Theorem \ref{associate}, we obtain  that
$$
\sup_{t>0} \frac{\int_0^t (u^{-1}\chi_I)^*_u (s) ds}{W^{1/t}(t)}\lesssim \frac{|I|}{W^{1/p}(u(I))}. 
$$
Now, 
\begin{eqnarray*}
\sup_{t>0} \frac{\int_0^t (u^{-1}\chi_I)^*_u (s) ds}{W^{1/t}(t)}&=&
\sup_{t>0} \frac 1{W^{1/p}(t)} \sup_{u(E)=t} \int_E \chi_I(x) dx
\\
&=&
\sup_{E\subset I} \frac{|E|}{W^{1/p}(u(E)}, 
\end{eqnarray*}
and the result follows. 

\noindent
{\it ii)} See Proposition 3.4.4 in \cite{crs:crs}.

\end{proof}

\begin{remark} 

\noindent
i) It was proved in \cite{crs:crs} that if $u=1$,  (\ref{GenHardyNec}) implies $w\in B_{p, \infty}$ and if $w=1$, then 
 (\ref{GenHardyNec}) implies  $u\in A_p$.  In fact, if $w(t)=t^\alpha$ with $\alpha>0$, then  (\ref{GenHardyNec}) also implies  $u\in A_p$. To see this, we observe that by hypothesis 
$$
\left( \int_0^{u(I)}\left( \frac{\phi_I(t)}{t^{\alpha+1}}\right)^{p'} t^{\alpha}dt \right)^{1/p'} 
\lesssim \frac{|I|}{u(I)^{(\alpha+1)/p}}.
$$
Then, if  $\gamma=\alpha (p'-1)$, 
\begin{align*}
& \left(\int_0^{u(I)} \left(\frac{\phi_I(t)}{t}\right)^{p'}dt \right)^{1/p'}\lesssim \left(\int_0^{u(I)} \left(\frac{\phi_I(t)}{t}\right)^{p'}\left(\frac{u(I)}{t}\right)^{\gamma} dt \right)^{1/p'} \\
&=u(I)^{\gamma/p'}\left( \int_0^{u(I)}\left( \frac{\phi_I(t)}{t^{\alpha+1}}\right)^{p'} t^{\alpha}dt \right)^{1/p'}
\lesssim \frac{|I|}{u^{1/p}(I)},\\
\end{align*}
and hence we obtain the condition for the case $w=1$ previously studied. 

\medskip

\noindent
ii) In fact, the above result can be generalized as follows: If  $W(t)/t$ is quasi-increasing and $w\in B^*_\infty$, then  (\ref{GenHardyNec})  implies $u\in A_p$. 
To see this we first observe that if $w\in B^*_\infty$ then, for every $f$ decreasing, 
$$
\int_0^\infty f(s) \frac{W(s)}s ds \lesssim \int_0^\infty f(s) w(s) ds
$$
and hence
\begin{eqnarray*} 
& &\left( \int_0^{u(I)}\left( \frac{\phi_I(t)}{t}\right)^{p'}dt \right)^{1/p'}= 
\left( \int_0^{u(I)}\left( \frac{\phi_I(t)}{t}\right)^{p'}\frac t{W(t)} \frac{W(t)}tdt \right)^{1/p'}
\\
&\lesssim& 
\left( \int_0^{u(I)}\left( \frac{\phi_I(t)}{t}\right)^{p'}\frac t{W(t)}  w(t)dt \right)^{1/p'}
\\
&=&\left( \int_0^{u(I)}\left( \frac{\phi_I(t)}{W(t)}\right)^{p'}\bigg(\frac {W(t)}{t}\bigg)^{p'-1}  w(t)dt \right)^{1/p'}
\\
&\lesssim&
\bigg(\frac {W(u(I))}{u(I)}\bigg)^{1/p} \frac{|I|}{W(u(I))^{1/p}}=  \frac{|I|}{u(I)^{1/p}}, 
\end{eqnarray*}
from which the result follows. 

\medskip

\noindent
iii) For every interval $I$,  we have that 

$$
\phi_I(t)=|I|-\psi^{-1}_I (u(I)-t), 
$$

\ 

\noindent
where $\psi_I(t)=\sup\{u(F): F\subset I \text{ and } |F|=t\}$. To see this, we first observe that 
$$
\phi_I(t)=\sup\{|F|: F\subset  [0, |I|]  \text{ and } (u\chi_I)^*(F)=t\}, 
$$
where $(u\chi_I)^*(F)=\int_F (u\chi_I)^*(s)ds$. Since $(u\chi_I)^*$ is decreasing we get that 
the above supremum is attained in a set $F=(a , |I|)$,  for some  $a>0$. 

Now, 
$$
t=\int_a^{|I|} (u\chi_I)^*(s)ds= \psi_I (|I|)- \psi_I (a)=u(I)-  \psi_I (a)
$$
and hence $a= \psi_I^{-1}( u(I)-t)$, from which the result follows. 

\end{remark}

\section{Consequences and Applications}

As a consequence of \eqref{conocidodiagonal}, some necessary conditions on  $p$, depending on $w$, were obtained in  \cite{crs:crs}. Following their approach,
we see that the same results can be obtained if we assume  the boundedness of the Hilbert transform on weighted Lorentz spaces. First, we need to define the index $p_w$:

\begin{definition}
Let $0<p<\infty$. We define
$$
p_w=\inf \left\{ p>0: \frac{t^p}{W(t)} \in L^{p'-1} \left((0,1), \frac{dt}{t} \right)  \right\},
$$
where $p'= \infty$, if $0 < p \leq 1$.
\end{definition}

\begin{proposition} Let $0<p<\infty$ and assume that $H:\Lambda^p_u(w) \to \Lambda^{p, \infty}_u(w)$ is bounded. Then $p\geq p_{w}$. Moreover,  if $p_{w}>1$ then $p> p_{w}$.
\end{proposition}

\begin{proof}
See the proof of Theorem 3.4.2  in \cite{crs:crs}.
\end{proof}

Another  important consequence of the fact that by Theorem \ref{NC} {\it (a)}  we can assume that both $u$ and $w$ are not integrable functions, is that ${\mathcal{C}}^{\infty}_c$ is dense in $\Lambda^p_u(w)$ as the following theorem shows:

\begin{theorem} \label{realdensity}
If $u  \notin L^1(\mathbb R)$, $w  \notin L^1(\mathbb R^+)$  and $w\in \Delta_2$, then ${\mathcal{C}}^{\infty}_c(\R)$  is  dense in $\Lambda^p_u(w)$.
\end{theorem}

\begin{proof} Let ${\mathcal{S}}_c(\R)$ be the space of simple functions with compact support and let us 
note that ${\mathcal{S}}_c(\R)$ is dense in $\Lambda^p_u(w)$.  Indeed, by Proposition~\ref{properties}~(c), we have that ${\mathcal{S}}_0(u)$
is dense in $\Lambda^p_u(w)$. On the other hand, given $f\in {\mathcal{S}}_0(u)$,  the sequence $f_{n}=f\chi_{(-n,n)}\in {\mathcal{S}}_c(\R)$ tends to  $f$ pointwise and hence, by Proposition \ref{properties} (b),  it also converges to $f$ in the quasi-norm $||\cdot||_{\Lambda^p_u(w)}$.

Now, to prove the density of ${\mathcal{C}}^{\infty}_c(\R)$ in ${\mathcal{S}}_c(\R)$ with respect to the topology induced by the quasi-norm of $\Lambda^p_u(w)$, it is enough to show that a characteristic function of a bounded measurable set can be approximated by smooth
functions of compact support. Thus, let $E$ be a bounded measurable set and let $\varepsilon>0$. Take a compact set $K\subset \R$ and a bounded open set $U\subset \R$ such that
$$
K\subset E\subset U \qquad \text{and} \qquad u(U\setminus K)\leq \delta ,
$$
for some small $\delta$ to be chosen.  Then, by Urysohn's lemma, there exists a function $f\in {\mathcal{C}}^\infty_c(\R)$
such that $\left. f\right|_K=1$, $\left. f\right|_{U^c}=0$, and $0\leq f\leq 1$. Then, since $|\chi_E-f|\leq {\chi}_{U \setminus K}$,
we get
$$
||{\chi}_{E}-f||_{\Lambda^p_u(w)}^p \leq ||{\chi}_{U\setminus K}||^p_{\Lambda^p_u(w)}=\int_{0}^{u(U\setminus K)}w(x)\,dx \leq \int_{0}^{\delta}w(x)\,dx.
$$
Therefore, choosing $\delta$ small enough we obtain that 
$$
||{\chi}_{E}-f||_{\Lambda^p_u(w)}\leq \varepsilon.
$$
\end{proof}

\begin{remark} Since $H$ is well defined on functions of ${\mathcal{C}}^{\infty}_c(\R)$, we deduce by  standard methods that the weak-type boundedness of $H^*$ on $\Lambda^p_u(w)$ implies that, for every $f\in \Lambda^p_u(w)$, the limit 
$$
\lim_{\varepsilon\to 0^+} \int_{|x-y| > \varepsilon} \frac{f(y)}{x-y}\,dy
$$
exists  almost everywhere.
\end{remark}

The third consequence of the previous results is the complete characterization of the weak boundedness of $H$ on $\Lambda^p_u(w)$ in the case $u\in A_1$.  We have to mention here that, for the case of the Hardy-Littlewood maximal operator $M$,  it was proved in \cite{crs:crs} that
$$
M: \Lambda^p_u(w) \to \Lambda^{p, \infty}_u(w)  \Longleftrightarrow  M: \Lambda^p(w) \to \Lambda^{p, \infty}(w).
$$
The following theorem shows that the same kind of result occurs for $H$. 
\begin{theorem}\label{a1}
Let  $u\in A_1$ and let $0<p<\infty$. Then

$$
H: \Lambda^p_u(w) \to \Lambda^{p, \infty}_u(w)  \Longleftrightarrow  w\in B_{p,\infty}\cap B^*_{\infty} .
$$
\end{theorem}

\begin{proof} Let us start by proving the necessary condition.  Since, by Theorem \ref{NC},  (\ref{conocidodiagonal}) holds, we can follow  the same argument used in \cite[ Proposition 3.4.4 and Theorem 3.4.8]{crs:crs} to conclude that $w\in B_{p,\infty}$. Let us see now that it is also in $B^*_{\infty}$.

Let $0<t \leq s<\infty$. Then, since $u\notin L^1(\mathbb R)$, there exists $\nu \in (0,1]$ and $b>0$  such that
$$
t=\int_{-b\nu}^{b\nu}u(r)\,dr \leq \int_{-b}^{b}u(r)\,dr =s.
$$
By Lemma \ref{necSteinWeiss} we obtain \eqref{log} and hence 
$$
\frac{W(t)}{W(s)}\lesssim \left(1+\log \frac{1}{\nu}\right)^{-p}.
$$
Let $S=(-b\nu, b\nu)$ and $I=(-b,b)$.
Since $u\in A_1$, we obtain that 
$$\nu=\frac{|S|}{|I|}\lesssim    \frac{u(S)}{u(I)} =  \frac{t}{s} 
$$ and therefore 
$$
\frac{W(t)}{W(s)}\lesssim \left(1+\log \frac{s}{t}\right)^{-p} ,
$$

From here, it follows that the function 
$$
\bar W(\lambda)=\sup_{s>0}\frac{W(\lambda s)}{W(s)}, \qquad 0<\lambda<1, 
$$
is a submultiplicative  function satisfying
$$
\bar W(\lambda)\lesssim \left(1+\log \frac{1}{\lambda}\right)^{-p}
$$
and taking $k\in\mathbb N$ such that $kp>1$, we obtain that
$$
\bar W(\lambda) \le \bar W(\lambda^{1/k})^k\lesssim \left(1+\log \frac{1}{\lambda}\right)^{-kp}.
$$
Consequently, 
\begin{eqnarray*}
\int_0^r \frac{W(t)} t dt &\lesssim& W(r) \int_0^r \left(1+\log \frac{r}{t}\right)^{-kp}  \frac{dt }t\\
&=& W(r) \int_1^\infty  \left(1+\log u\right)^{-kp}  \frac{du }u
\lesssim W(r),
\end{eqnarray*}
and hence, by (\ref{binfty}),   $w\in B^*_{\infty}$.

To prove the converse, we just have to use that if $u\in A_1$, then   (see~\cite{bk:bk, br:br}):
\begin{equation}\label{vv}
(H^*f)^*_u (t)\le \frac 1t \int_0^t f^*_u(s) ds + \int_t^\infty f^*_u(s)\frac{ds}s:= P f^*_u(t)+ Q f^*_u(t).
\end{equation}
whenever the right hand side is finite. 

Now, since $w\in B_{p, \infty}$, we have that
$$
\sup_{t>0}P f^*_u(t) W(t)^{1/p}\lesssim ||f^*_u||_{L^p(w)}=||f||_{\Lambda^p_u(w)},
$$
 and the condition  $w\in B^*_\infty$  implies the same inequality for the operator $Q$; that is (see \cite{k:k}, \cite{n:n})
 $$
\sup_{t>0}Q f^*_u(t) W(t)^{1/p}\lesssim ||f||_{\Lambda^p_u(w)}. 
$$
Therefore, we obtain that 
$$
H^*: \Lambda^p_u(w) \to \Lambda^{p, \infty}_u(w) 
$$
is bounded and  by Fatou's lemma we obtain the result. 
  \end{proof}
  
  With a very similar  proof and using the properties of the class $B_p$, the following result is obtained. 
  
  \begin{theorem}\label{a2}
Let  $u\in A_1$ and let $1<p<\infty$. Then

$$
H: \Lambda^p_u(w) \to \Lambda^{p}_u(w)  \Longleftrightarrow  w\in B_{p}\cap B^*_{\infty} .
$$
\end{theorem}

In the case $p\le 1$, we  obtain the following: 

\begin{theorem}\label{a3}
Let  $u\in A_1$ and let $0<p\le 1$. Then, if  $w\in B_{p}\cap B^*_{\infty}$, we have that 
$$
H: \Lambda^p_u(w) \to \Lambda^{p}_u(w)
$$
is bounded. 
\end{theorem}

\begin{remark} \label{maxim} Observe than from the proofs of Theorem \ref{a1} and \ref{a2}, we also have that, if $u\in A_1$, then 
$$
H^*: \Lambda^p_u(w) \to \Lambda^{p, \infty}_u(w)  \Longleftrightarrow  w\in B_{p,\infty}\cap B^*_{\infty}, \qquad p>0, 
$$
and
$$
H^*: \Lambda^p_u(w) \to \Lambda^{p}_u(w)  \Longleftrightarrow  w\in B_{p}\cap B^*_{\infty}, \qquad p>1.
$$

\end{remark}

\ 

Our next application concerns the boundedness of $H$ on the Lorentz spaces $L^{p,q}(u)$. 
In \cite{chk:chk}, Chung, Hunt, and Kurtz  provided a sufficient condition for the boundedness of
$$
H: L^{p,1}(u) \to L^{p,\infty}(u), \qquad 1\le p<\infty. 
$$
We prove that this condition is  necessary and also complete the characterization of the above boundedness for some other exponents.

\begin{theorem}
Let $p,r,  \in (0,\infty), \,\, q,s \in (0,\infty]$.

\begin{enumerate}
\item[{(i)}] 
If $p<1$ or  $p\neq r$, there are no  doubling  weights $u$ such that $H: L^{p,q}(u) \to L^{r,s}(u)$ is bounded.

\item[{(ii)}] If $q\leq 1$, the boundedness
$$
H: L^{1,q}(u) \to L^{1,\infty}(u)
$$
holds if and only if  $u \in A_1$.

\item[{(iii)}] If $q>1$, the boundedness
$$
H: L^{1,q}(u) \to L^{1,\infty}(u)
$$
does not hold for any $u$. 
\end{enumerate}
\end{theorem}

\begin{proof} The proofs follow the same ideas of \cite[Theorem 3.5.1]{crs:crs}:
\medskip

\noindent (i)
 Since $L^{r,s}(u) \subset L^{r,\infty}(u)$, if $H: L^{p,q}(u) \to L^{r,s}(u)$ we would have that  $H: L^{p,q}(u) \to L^{r,\infty}(u)$ is bounded, which is equivalent to have that 
$ H:\Lambda^{q}_{u} (t^{q/p-1}) \to \Lambda^{q,\infty}_{u} (t^{q/r-1}) $ is bounded. Arguing as in Proposition \ref{necesaria geometrica},    we obtain that 
\begin{equation}\label{expo}
\frac{u^{1/r}(I)}{|I|} \lesssim \frac{u^{1/p}(E)}{|E|},  \,\,\,\,\, E\subset I. 
\end{equation}
 Then, by the  Lebesgue differentiation theorem we get  first that $p\geq 1$.
On the other hand, if  we take $E=I$, then \eqref{expo} implies $u^{1/r-1/p}(I)\lesssim 1$ and hence $p=r$, since $u\not \in L^1$ by Theorem \ref{NC} {\it (a)}. 
\medskip

\noindent (ii)   As before, the boundedness $H: L^{1,q}(u) \to L^{1,\infty}(u)$ can be rewritten as
$H:\Lambda^{q}_{u} (t^{q-1}) \to \Lambda^{q,\infty}_{u} (t^{q-1})$, and by (\ref{quasiconcrara}) we get
$$
\frac{u(I)}{|I|} \lesssim \frac{u(E)}{|E|},  \,\,\,\,\, E\subset I, 
$$
which is equivalent to the condition $A_1$.
On the other hand if $u\in A_1$ then $H:\Lambda^{q}_{u} (t^{q-1}) \to \Lambda^{q,\infty}_{u} (t^{q-1})$ holds by
Theorem \ref{a1}, taking into account that, since $q\le 1$, the weight $t^{q-1}$ satisfies the condition $B_{q, \infty}\cap B^*_{\infty}$. 
\medskip

\noindent
(iii) If $
H: L^{1,q}(u) \to L^{1,\infty}(u)
$ is bounded, we also have  the boundedness of $
H: L^{1,r}(u) \to L^{1,\infty}(u)
$, 
 $r<1$ and thus, 
 by (ii),  we  have that $u\in A_1$. But in this case, Theorem \ref{a2} asserts that   $w$ must be  in $B_q$, while   $t^{q-1}\notin B_q$,  if $q>1$. 
\end{proof}

For the case $p>1$ we finally obtain the following necessary conditions. 

\begin{proposition}

  Let $p>1$ and $0<q\leq s\leq \infty$. If
$$
H: L^{p,q}(u) \to L^{p,s}(u)
$$
is bounded then:

\noindent
\begin{enumerate}
\item[{(1)}] 
 Case $q\leq 1$ and $s=\infty$:\, $\displaystyle{\frac{u(I)}{|I|^{p}} \lesssim \frac{u(E)}{|E|^p}, \,\,\,\,\,
E\subset I}$.

\noindent
\item[{(2)}]  Case $q>1$ or $s<\infty$:\, $u\in A_p$.
\end{enumerate}
\end{proposition}

\begin{proof}

\noindent (1)  If $q \leq 1$ and $s=\infty$,  $H: L^{p,q}(u) \to L^{p,\infty}(u)$
can be rewritten as $H:\Lambda^{q}_{u} (t^{q/p-1}) \to \Lambda^{q,\infty}_{u} (t^{q/p-1})$, which by (\ref{quasiconcrara}) implies $\displaystyle{\frac{u(I)}{|I|^{p}} \lesssim \frac{u(E)}{|E|^p}, \,\,\,\,\, E\subset I}$ as we wanted to see. 
\medskip

\noindent (2)   Let $q>1$ or $s<\infty$. By (\ref{conocidodiagonal}), the boundedness  $H: L^{p,q}(u) \to L^{p,s}(u)$ implies that
  
$$
||u^{-1}\chi_{I}||_{L^{p',q'}(u)} ||\chi_{I}||_{L^{p,q}(u)} \lesssim |I|,
$$

\ 

\noindent
for all intervals $I\subset\R$. But in \cite{chk:chk}  it was proved  that this  agrees with   
 the $A_p$ condition. 
\end{proof}

\end{document}